\documentclass[12pt,twoside]{amsart}
\usepackage{cases}
\usepackage{amscd}
\usepackage{amsmath}
\usepackage{amsthm}
\usepackage{amsfonts}
\usepackage{amssymb}
\usepackage{latexsym}
\usepackage[all]{xy}
\usepackage{epsfig}

\date{}
\allowdisplaybreaks[4] \footskip=15pt
\renewcommand{\eqnarray*}
\renewcommand{\uppercasenonmath}[1]{}

\numberwithin{equation}{section} \theoremstyle{plain}
\newtheorem*{thm*}{Main Theorem}
\newtheorem{thm}{Theorem}[section]
\newtheorem{cor}[thm]{Corollary}
\newtheorem*{cor*}{Corollary}
\newtheorem{lem}[thm]{Lemma}
\newtheorem*{lem*}{Lemma}
\newtheorem{prop}[thm]{Proposition}
\newtheorem*{prop*}{Proposition}
\newtheorem{rem}[thm]{Remark}
\newtheorem*{rem*}{Remark}

\newtheorem*{exa*}{Example}
\newtheorem{df}[thm]{Definition}
\newtheorem*{df*}{Definition}

\newtheorem*{conj*}{Conjecture}
\newtheorem*{ack*}{ACKNOWLEDGEMENTS}

\newcommand{\pf}{\noindent\begin {proof}}
\newcommand{\epf}{\end{proof}}

\begin{document}
\begin{center}
{\Large \bf Cartan-Eilenberg complexes and Auslander categories
\footnotetext{

E-mail: renwei@nwnu.edu.cn (W. Ren), liuzk@nwnu.edu.cn (Z.K. Liu).
}}

\vspace{0.5cm}    Wei Ren and Zhongkui Liu\\
{\small\it Department of Mathematics, Northwest Normal University, Lanzhou 730070, China. \\

}
\end{center}

\bigskip
\centerline { \bf  Abstract}
\leftskip6truemm \rightskip6truemm \noindent
Let $R$ be a commutative noetherian ring with a semi-dualizing module $C$. The Auslander categories with respect to $C$ are related through
Foxby equivalence: $\xymatrix@C=50pt{\mathcal
{A}_C(R) \ar@<0.4ex>[r]^{~C\otimes^{\mathbf{L}}_{R} -} & \mathcal {B}_C(R)
   \ar@<0.4ex>[l]^{~\mathbf{R}\mathrm{Hom}_{R}(C, -)}}$. We firstly intend to extend the Foxby equivalence
to Cartan-Eilenberg complexes. To this end, C-E Auslander categories, C-E $\mathcal{W}$ complexes and
C-E $\mathcal{W}$-Gorenstein complexes are introduced, where $\mathcal{W}$ denotes a self-orthogonal class of $R$-modules.
Moreover, criteria for finiteness of C-E Gorenstein dimensions of complexes in terms of resolution-free characterizations
are considered.  \\
\vbox to 0.3cm{}
\vspace{4mm}

\leftskip0truemm \rightskip0truemm \vbox to 0.2cm{}

\section{\bf Introduction}
In \cite{V97} Verdier introduced the notions of a Cartan-Eilenberg projective (injective) complex and
a Cartan-Eilenberg projective (injective) resolution, which originated from \cite{CE56}. A complex $P$ is said
to be Cartan-Eilenberg projective provided that $P$, $\mathrm{Z}(P)$, $\mathrm{B}(P)$ and $\mathrm{H}(P)$
are all complexes of projective modules.
Recently, Enochs showed in \cite{E11} that Cartan-Eilenberg resolutions can be defined in terms
of precovers and preenvelopes by Cartan-Eilenberg projective and injective complexes, and he further
introduced Cartan-Eilenberg Gorenstein projective and injective complexes.

Let $R$ be a commutative noetherian ring.
Recall that a semi-dualizing module over $R$, which provides a generalization of $R$ and a dualizing module, is a
finite generated module $C$ such that the homothety morphism
$R\rightarrow \mathrm{\mathbf{R}Hom}_{R}(C, C)$ is invertible in the derived category $\mathcal{D}(R)$.
For a given semi-dualizing module $C$, consider two triangulated subcategories
of the bounded derived category $\mathcal{D}_{b}(R)$, namely the Auslander class
 $\mathcal{A}_{C}(R)$ and the Bass class $\mathcal{B}_{C}(R)$.
The adjoint pair of functors $(C\otimes^{\mathbf{L}}_{R} -, \mathbf{R}\mathrm{Hom}_{R}(C, -))$ on $\mathcal{D}(R)$
restricts to a pair of equivalences of
categories
$\xymatrix@C=80pt{\mathcal
{A}_C(R) \ar@<0.4ex>[r]^{~C\otimes^{\mathbf{L}}_{R} -} & \mathcal {B}_C(R)
   \ar@<0.4ex>[l]^{~\mathbf{R}\mathrm{Hom}_{R}(C, -)}}$
known as Foxby equivalence.

Using \cite[Proposition 4.4]{Chr01}, it is direct to obtain a commutative diagram:
\begin{center}$\xymatrix@C=90pt{
\overline{\mathcal{P}}(R)\ar@{^{(}->}[d]\ar@<0.4ex>[r]^{C\otimes^{\mathbf{L}}_{R}-}
&\overline{\mathcal{P}}_{C}(R)\ar@{^{(}->}[d] \ar@<0.4ex>[l]^{\mathbf{R}\mathrm{Hom}_{R}(C, -)}\\
 \mathcal{A}_{C}(R) \ar@<0.4ex>[r]^{C\otimes^{\mathbf{L}}_{R}-}
&\mathcal{B}_{C}(R) \ar@<0.4ex>[l]^{\mathbf{R}\mathrm{Hom}_{R}(C, -)}\\
\overline{\mathcal{I}}_{C}(R) \ar@{^{(}->}[u] \ar@<0.4ex>[r]^{C\otimes^{\mathbf{L}}_{R}-}
&\overline{\mathcal{I}}(R)\ar@{^{(}->}[u]\ar@<0.4ex>[l]^{\mathbf{R}\mathrm{Hom}_{R}(C, -)}
 }$\end{center}
By $\overline{\mathcal{P}}(R)$ (resp. $\overline{\mathcal{P}}_{C}(R)$) we denote the category consisting of complex $X$
such that $X\simeq U$ in $\mathcal{D}_{b}(R)$, where $U$ is a bounded complex
of projective (resp. $C$-projective) modules; and similarly $\overline{\mathcal{I}}(R)$ and $\overline{\mathcal{I}}_{C}(R)$ are defined.

So it is natural to ask: whether there exist subcategories of $\mathcal{D}_{b}(R)$ with respect to
Cartan-Eilenberg (C-E for short) complexes, for which one has the related Foxby equivalence?
In section 4, we introduce C-E Auslander class $\mathrm{CE}\text{-}\mathcal{A}_{C}(R)$ and C-E Bass class $\mathrm{CE}\text{-}\mathcal{B}_{C}(R)$,
and obtain the following diagram by collecting separating results for proving it, which extends the existed Foxby equivalence.
We remark that the dual of the diagram exists as well.

\vspace{2mm} \noindent{\bf Theorem A.}\label{thm A}\textit{
Let $R$ be a commutative noetherian ring. If $C$ is a semi-dualizing module for $R$, then there is a commutative diagram, where the vertical inclusions
are full embedding, and the horizontal arrows are equivalences of categories:
$$\xymatrix@C=100pt{
\mathrm{CE}\text{-}\overline{\mathcal{P}}(R) \ar@{^{(}->}[d]\ar@<0.4ex>[r]^{C\otimes^{\mathbf{L}}_{R}-}
&\mathrm{CE}\text{-}\overline{\mathcal{P}_{C}}(R) \ar@{^{(}->}[d] \ar@<0.4ex>[l]^{\mathbf{R}\mathrm{Hom}_{R}(C, -)}\\
\mathrm{CE}\text{-}\overline{\mathcal{G}(\mathcal{P})}(R)\cap \mathrm{CE}\text{-}\mathcal{A}_{C}(R) \ar@{^{(}->}[d]\ar@<0.4ex>[r]^{C\otimes^{\mathbf{L}}_{R}-}
&\mathrm{CE}\text{-}\overline{\mathcal{G}(\mathcal{P}_C)}(R) \ar@{^{(}->}[d]\ar@<0.4ex>[l]^{\mathbf{R}\mathrm{Hom}_{R}(C, -)}\\
\mathrm{CE}\text{-}\mathcal{A}_{C}(R) \ar@{^{(}->}[d] \ar@<0.4ex>[r]^{C\otimes^{\mathbf{L}}_{R}-}
&\mathrm{CE}\text{-}\mathcal{B}_{C}(R) \ar@{^{(}->}[d]\ar@<0.4ex>[l]^{\mathbf{R}\mathrm{Hom}_{R}(C, -)}\\
\mathcal{A}_{C}(R)  \ar@<0.4ex>[r]^{C\otimes^{\mathbf{L}}_{R}-}
&\mathcal{B}_{C}(R) \ar@<0.4ex>[l]^{\mathbf{R}\mathrm{Hom}_{R}(C, -)}}
$$}

Since the complexes are related to the classes of projective, injective, $C$-projective and
$C$-injective modules, which are known to be self-orthogonal with respect to Ext, in section 2
we generally define and study C-E $\mathcal{W}$ complexes relative to
a self-orthogonal class $\mathcal{W}$ of $R$-modules. In section 3, we focus on studying C-E
$\mathcal{W}$-Gorenstein complexes by showing that two potential choices for defining these complexes are equivalent.

\vspace{2mm} \noindent{\bf Theorem B.}\label{thm B}\textit{
For a complex $G$, the following are equivalent:\\
\indent $(1)$ $G$ is a $\mathrm{C}$-$\mathrm{E}$ $\mathcal{W}$-Gorenstein complex,
i.e. $G$, $\mathrm{B}(G)$, $\mathrm{Z}(G)$ and $\mathrm{H}(G)$ are complexes of $\mathcal{W}$-Gorenstein modules.\\
\indent (2) $G$ admits a $\mathrm{C}$-$\mathrm{E}$ complete $\mathcal{W}$ resolution (see Definition \ref{df 3.4}).}\\

The notions of a dualizing complex and a dualizing module have
been extensively developed and applied in many areas such as commutative algebra and algebraic geometry.
Recall that a semi-dualizing module $C$ is dualizing provided that $C$ has finite injective dimension.
Auslander categories with respect to semi-dualizing and dualizing modules (complexes) are used to find criteria for finiteness of
Gorenstein homological dimensions of modules
and complexes in terms of resolution-free characterizations, see for examples \cite{Chr00, Chr01, CFH06, HJ06, SWSW10}.
Finally, homological dimensions with respect to C-E complexes are considered in section 5, and we obtain the following.
Moreover, it is worth noting that the complexes considered in Theorem A are precisely those with finite C-E homological dimensions.

\vspace{2mm} \noindent{\bf Theorem C.}\label{thm C}\textit{
Let $R$ be a commutative noetherian ring with a dualizing module $C$, $X\in\mathcal{D}_{\sqsupset}(R)$.
Assume that $X\simeq G$ for a $\mathrm{C}$-$\mathrm{E}$ Gorenstein projective complex $G$.
Then the following are equivalent:\\
\indent$(1)$  $\mathrm{CE}\text{-}\mathrm{Gpd}_{R}X$ is finite.\\
\indent$(2)$  $X\in \mathrm{CE}\text{-}\mathcal{A}_{C}(R)$.\\
\indent$(3)$  $\mathrm{Gpd}_{R}X$, $\mathrm{Gpd}_{R}\mathrm{Z}(X)$, $\mathrm{Gpd}_{R}\mathrm{B}(X)$ and $\mathrm{Gpd}_{R}\mathrm{H}(X)$ are all finite.\\
\indent$(4)$  $X$, $\mathrm{Z}(X)$, $\mathrm{B}(X)$ and $\mathrm{H}(X)$ are in $\mathcal{A}_{C}(R)$.
}\\

\noindent{\bf Notations.}  Throughout, $R$ is a commutative ring.
An $R$-complex
$X = \cdots\longrightarrow X_{n+1}\stackrel{d_{n+1}}\longrightarrow X_{n}\stackrel{d_{n}}\longrightarrow X_{n-1}\longrightarrow\cdots$
will be denoted $(X_{i}, d_{i})$ or $X$. The $n$th cycle module is defined as $\mathrm{Ker}d_{n}$ and is denoted by $\mathrm{Z}_{n}(X)$,
$n$th boundary module is $\mathrm{Im}d_{n+1}$ and is denoted by $\mathrm{B}_{n}(X)$,
and $n$th homology module is $\mathrm{H}_{n}(X) = \mathrm{Z}_{n}(X) / \mathrm{B}_{n}(X)$. The complexes of cycles and boundaries, and the
homology complex of $X$ is denoted by $\mathrm{Z}(X)$, $\mathrm{B}(X)$ and $\mathrm{H}(X)$ respectively.
For a given $R$-module $M$, we let $S^{n}(M)$ denote the complex with all entries 0 except $M$ in degree $n$. We let
$D^{n}(M)$ denote the complex $X$ with $X_{n} = X_{n-1} = M$ and all other entries 0,
and with all maps 0 except $d_{n} = 1_{M}$.

We use $\mathcal{C}(R)$ to denote the category of $R$-complexes.
The symbol ``$\simeq$'' is used to designate quasi-isomorphisms
in the category $\mathcal{C}(R)$ and isomorphisms in the derived category $\mathcal{D}(R)$.
The left derived functor of the tensor product functor of
$R$-complexes is denoted by $-\otimes ^{\mathbf{L}}_{R}-$, and $\mathbf{R}\mathrm{Hom}_{R}(-, -)$
denotes the right derived functor of the homomorphism functor of complexes.

Let $C$ be a semi-dualizing module over a noetherian ring $R$.
The Auslander categories with respect to $C$,
denoted by $\mathcal{A}_{C}(R)$ and $\mathcal{B}_{C}(R)$, are the full subcategories of $\mathcal{D}_{b}(R)$
whose objects are specified as follows:
$$\mathcal{A}_{C}(R)= \left\{X\in \mathcal{D}_{b}(R) \, \vline \,\begin{matrix}  \eta_{X}: X\rightarrow
 \, \mathbf{R}\mathrm{Hom}_{R}(C,  C\otimes_{R}^{\mathbf{L}}X) \text{ is an iso-}\\
 \text{morphism in }\mathcal{D}(R),
 \text{ and } C\otimes_{R}^{\mathbf{L}}X\in \mathcal{D}_{b}(R)
 \end{matrix} \right\}
$$
and
$$\mathcal{B}_{C}(R)= \left\{Y\in \mathcal{D}_{b}(R) \, \vline \,\begin{matrix}
\varepsilon_{Y}: C\otimes_{R}^{\mathbf{L}}\mathbf{R}\mathrm{Hom}(C, Y)\rightarrow Y
\text{ is an isomor-} \\ \text{phism in }\mathcal{D}(R), \text{ and }\mathbf{R}\mathrm{Hom}(C, Y)
\in \mathcal{D}_{b}(R)
 \end{matrix} \right\}
,$$
where $\eta$ and $\varepsilon$ denote the unit and counit of the adjoint pair
$(C\otimes_{R}^{\mathbf{L}}-, \mathbf{R}\mathrm{Hom}(C, -))$.

\section{\bf Cartan-Eilenberg $\mathcal{W}$ complexes}

Let $\mathcal{W}$ be a class of $R$-modules. $\mathcal{W}$ is called self-orthogonal if it satisfies
the following condition:
$$\mathrm{Ext}^{i}_{R}(W, W^{'}) = 0 \text{ for all  }  W, W^{'}\in \mathcal{W} \text{ and all } i \geq 1.$$
In the following, $\mathcal{W}$ always denotes a self-orthogonal class of $R$-modules which is closed
under extensions, finite direct sums and direct summands. Geng and Ding enumerated
in \cite[Remark 2.3]{GD11} some interesting examples of self-orthogonal classes.

\begin{df}\label{df 2.1}
A complex $X$ is said to be a $\mathrm{C}$-$\mathrm{E}$ $\mathcal{W}$ complex if $X$, $\mathrm{Z}(X)$, $\mathrm{B}(X)$
and $\mathrm{H}(X)$ are complexes each of whose terms belongs to $\mathcal{W}$.
\end{df}

\begin{rem}\label{rem 2.2}
\indent $(1)$ For any module $M\in \mathcal{W}$ and any $n\in \mathbb{Z}$, $S^{n}(M)$ and $D^{n}(M)$ are $\mathrm{C}$-$\mathrm{E}$ $\mathcal{W}$ complexes.\\
\indent $(2)$ In particular, if $\mathcal{W}$ denotes the class of projective (resp. injective) modules, then $\mathrm{C}$-$\mathrm{E}$ $\mathcal{W}$ complexes
are precisely $\mathrm{C}$-$\mathrm{E}$ projective (resp. $\mathrm{C}$-$\mathrm{E}$ injective) complexes.
\end{rem}

We will frequently consider complexes $X$ with $d^{X} = 0$. Such a complex is completely determined by its family of terms
$(X_{n})_{n\in \mathbb{Z}}$, i.e. by the underlying structure of $X$ as a graded module with the grading over $\mathbb{Z}$.
So by the term graded module we will mean a complex $X$ with $d^{X} = 0$.

Recall that $X$ is called a $\mathcal{W}$ complex \cite{Liang} if $X$ is exact and Z$_{n}(X)\in \mathcal{W}$ for any
$n\in \mathbb{Z}$. We will denote the class of $\mathcal{W}$ complexes by $\widetilde{\mathcal{W}}$.

\begin{prop}\label{prop 2.3}
$X$ is a $\mathrm{C}$-$\mathrm{E}$  $\mathcal{W}$ complex if and only if $X$ can be divided into direct sums $X = X^{'} \oplus X^{''}$
where $X^{'}\in \widetilde{\mathcal{W}}$ and $X^{''}$ is a graded module with all items in $\mathcal{W}$.
\end{prop}

\begin{proof}
Since $X^{'}\in \widetilde{\mathcal{W}}$ is exact, B$_{n}(X^{'})$ = Z$_{n}(X^{'}) \in \mathcal{W}$, H$_{n}(X^{'}) = 0$
for all $n\in \mathbb{Z}$. Then $X^{'}$ is a C-E $\mathcal{W}$ complex. It is easy to see $X^{''}$ is  a C-E $\mathcal{W}$ complex.
Then such direct sum is a C-E $\mathcal{W}$ complex.

Conversely, suppose that $X$ is a C-E $\mathcal{W}$ complex. We have the exact sequences of $R$-modules
$0\rightarrow \mathrm{B}_{n}(X)\rightarrow \mathrm{Z}_{n}(X)\rightarrow \mathrm{H}_{n}(X)\rightarrow 0$ and
$0\rightarrow \mathrm{Z}_{n}(X)\rightarrow X_{n}\rightarrow \mathrm{B}_{n-1}(X)\rightarrow 0$.
Since $\mathrm{Z}_{n}(X), \mathrm{B}_{n}(X), \mathrm{H}_{n}(X)\in \mathcal{W}$ for all $n\in \mathbb{Z}$, each sequence splits.
This allows us to write $X_{n} = \mathrm{B}_{n}(X)\oplus\mathrm{H}_{n}(X)\oplus\mathrm{B}_{n-1}(X)$.
Then $$d_{n}: X_{n} = \mathrm{B}_{n}(X)\oplus\mathrm{H}_{n}(X)\oplus\mathrm{B}_{n-1}(X) \rightarrow
X_{n-1} = \mathrm{B}_{n-1}(X)\oplus\mathrm{H}_{n-1}(X)\oplus\mathrm{B}_{n-2}(X)$$
is the map $(x, y, z)\rightarrow (z, 0, 0)$. Let $X^{'} = \bigoplus_{n\in \mathbb{Z}}D^{n}(\mathrm{B}_{n-1}(X))$,
$X^{''} = \bigoplus_{n\in\mathbb{Z}}S^{n}(\mathrm{H}_{n}(X))$.
We then obtain the desired direct sum decomposition $X = X^{'} \oplus X^{''}$.
\end{proof}

\begin{cor}\label{cor 2.4}(\cite[Proposition 3.4]{E11})
A complex $X$ is $\mathrm{C}$-$\mathrm{E}$ projective if and only if $X$ can be divided into direct sums $X = X^{'} \oplus X^{''}$,
where $X^{'}$ is a projective complex and $X^{''}$ is a graded module with all items being projective.
\end{cor}

Let $R$ be a noetherian ring with a semi-dualizing module $C$.
An $R$-module is $C$-projective if it has the form $C\otimes_{R}P$ for some projective $R$-module $P$.
An $R$-module is $C$-injective if it has the form $\mathrm{Hom}_{R}(C, E)$ for some injective $R$-module $E$.
Let $\mathcal{P}_{C} = \{ C\otimes_{R}P \mid P\textrm{ is a projective }R\textrm{-module} \}$ and
$\mathcal{I}_{C} = \{ \mathrm{Hom}_{R}(C, E) \mid E\textrm{ is an injective }R\textrm{-module} \}$
denote the class of $C$-projective and $C$-injective modules, respectively.

By \cite[Theorem 3.1]{GD11},  $\mathcal{P}_{C} = \mathrm{Add}C$ and $\mathcal{I}_{C} = \mathrm{Prod}C^{+}$,
where $\mathrm{Add}C$ stands for the category consisting of all modules isomorphic to direct summands of direct sums of copies of $C$,
and $\mathrm{Prod}C^{+}$ the category consisting of all modules isomorphic to direct summands of direct products of copies of
$C^{+} = \mathrm{Hom}_{R}(C, Q)$ with $Q$ an injective cogenerator.
Thus $\mathcal{P}_{C}$ and $\mathcal{I}_{C}$ are self-orthogonal.
If we put $\mathcal{W}=\mathcal{P}_{C}$
(resp. $\mathcal{W}=\mathcal{I}_{C}$), then a $\mathrm{C}$-$\mathrm{E}$ $\mathcal{W}$ complex above is particularly
called a $\mathrm{C}$-$\mathrm{E}$ $C$-projective (resp. $\mathrm{C}$-$\mathrm{E}$ $C$-injective) complex.

\begin{cor}\label{cor 2.5}
Let $R$ be a noetherian ring with a semi-dualizing module $C$ and $X$ an $R$-complex.
Then $X$ is $\mathrm{C}$-$\mathrm{E}$ $C$-projective if and only if $X$ can be divided into direct sums $X = X^{'} \oplus X^{''}$
where $X^{'}\in \widetilde{\mathcal{P}_{C}}$ and $X^{''}$ is a graded module with all items
in $\mathcal{P}_{C}$.

Similarly, $X$ is $\mathrm{C}$-$\mathrm{E}$ $C$-injective if and only if $X$ can be divided into direct sums $X = X^{'} \oplus X^{''}$
where $X^{'}\in \widetilde{\mathcal{I}_{C}}$ and $X^{''}$ is a graded module with all items in $\mathcal{I}_{C}$.
\end{cor}

\section{\bf Cartan-Eilenberg $\mathcal{W}$-Gorenstein complexes}
Recall that an $R$-module $M$ is said to be $\mathcal{W}$-Gorenstein \cite[Definition 2.2]{GD11} if there exists an exact sequence
$$W_{\bullet} = \cdots\longrightarrow W_{1}\longrightarrow W_{0}\longrightarrow  W_{-1}\longrightarrow  W_{-2}\longrightarrow \cdots$$
of modules in $\mathcal{W}$ such that $M = \mathrm{Ker}(W_{-1} \rightarrow W_{-2})$ and $W_{\bullet}$ is
$\mathrm{Hom}_{R}(\mathcal{W},-)$ and $\mathrm{Hom}_{R}(-, \mathcal{W})$ exact. In this case, $W_{\bullet}$ is called a complete
$\mathcal{W}$-resolution of $M$.
This covers a various of examples by different choices of $\mathcal{W}$, for instance, Gorenstein projective and Gorenstein
injective modules.

We note that the class of $\mathcal{W}$-Gorenstein modules is just the class $\mathcal{G}(\mathcal{W})$ introduced
by Sather-Wagstaff and coauthors \cite{SWSW08}
when the abelian category $\mathcal{A}$ is taken to be the category of $R$-modules. So $\mathcal{G}(\mathcal{W}) = \mathcal{G}^{2}(\mathcal{W})$.
The symbol $\mathcal{G}^{2}(\mathcal{W})$ denotes the class of $R$-module $M$ defined by an iteration of the procedure used to define
$\mathcal{W}$-Gorenstein modules, that is,  there exists an exact sequence
$$G_{\bullet} = \cdots\longrightarrow G_{1}\longrightarrow G_{0}\longrightarrow G_{-1}\longrightarrow G_{-2}\longrightarrow \cdots$$
of $\mathcal{W}$-Gorenstein modules such that $M = \mathrm{Ker}(G_{-1} \rightarrow G_{-2})$ and $G_{\bullet}$  remains exact by applying
$\mathrm{Hom}_{R}(G,-)$ and $\mathrm{Hom}_{R}(-, G)$ for any $\mathcal{W}$-Gorenstein module $G$.

\begin{df}\label{df 3.1}
A complex $X$ is said to be a $\mathrm{C}$-$\mathrm{E}$ $\mathcal{W}$-Gorenstein complex if $X$, $\mathrm{Z}(X)$, $\mathrm{B}(X)$
and $\mathrm{H}(X)$ are complexes consisting of $\mathcal{W}$-Gorenstein modules.
\end{df}

We will show that one can also use a modification of the definition of $\mathcal{W}$-Gorenstein module to define such a complex.
First, we need to recall the following definition.

\begin{df}\label{df 3.2}(\cite[Definition 5.3]{E11})
A complex of complexes
$$\mathcal{C} = \cdots\longrightarrow C^{2}\longrightarrow C^{1}\longrightarrow  C^{0}\longrightarrow  C^{-1}\longrightarrow \cdots$$
is said to be $\mathrm{C}$-$\mathrm{E}$  exact if the following sequences are all exact:\\
\indent $(\mathrm{1})$ $\cdots\longrightarrow C^{1}\longrightarrow  C^{0}\longrightarrow  C^{-1}\longrightarrow \cdots$.\\
\indent $(\mathrm{2})$ $\cdots\longrightarrow \mathrm{Z}(C^{1})\longrightarrow  \mathrm{Z}(C^{0})\longrightarrow  \mathrm{Z}(C^{-1})\longrightarrow \cdots$.\\
\indent $(\mathrm{3})$ $\cdots\longrightarrow \mathrm{B}(C^{1})\longrightarrow  \mathrm{B}(C^{0})\longrightarrow  \mathrm{B}(C^{-1})\longrightarrow \cdots$.\\
\indent $(\mathrm{4})$ $\cdots\longrightarrow C^{1}/\mathrm{Z}(C^{1})\longrightarrow  C^{0}/\mathrm{Z}(C^{0})\longrightarrow  C^{-1}/\mathrm{Z}(C^{-1})\longrightarrow \cdots$.\\
\indent $(\mathrm{5})$ $\cdots\longrightarrow C^{1}/\mathrm{B}(C^{1})\longrightarrow  C^{0}/\mathrm{B}(C^{0})\longrightarrow  C^{-1}/\mathrm{B}(C^{-1})\longrightarrow \cdots$.\\
\indent $(\mathrm{6})$ $\cdots\longrightarrow \mathrm{H}(C^{1})\longrightarrow  \mathrm{H}(C^{0})\longrightarrow  \mathrm{H}(C^{-1})\longrightarrow \cdots$.
\end{df}

\begin{rem}\label{rem 3.3}
In the above definition, exactness of (1) and (2) implies exactness of all (1)-(6), and exactness of (1) and (5) implies
exactness of all (1)-(6).
\end{rem}

In the following, we focus on C-E $\mathcal{W}$-Gorenstein complexes and we show that such complexes can be obtained by a so-called
C-E complete $\mathcal{W}$-resolution.

\begin{df}\label{df 3.4}
For a complex $G$, by a  $\mathrm{C}$-$\mathrm{E}$ complete $\mathcal{W}$-resolution of $G$ we mean a
$\mathrm{C}$-$\mathrm{E}$ exact sequence
$$\mathbb{W} = \cdots\longrightarrow W^{1}\longrightarrow W^{0}\longrightarrow  W^{-1}\longrightarrow  W^{-2}\longrightarrow \cdots$$
of $\mathrm{C}\text{-}\mathrm{E }~\mathcal{W}$ complexes with $G = \mathrm{Ker}\,(W^{-1}\rightarrow W^{-2})$, such that it remains
exact after applying $\mathrm{Hom}_{\mathcal{C}(R)}(V, -)$ and $\mathrm{Hom}_{\mathcal{C}(R)}(-, V)$
for any $\mathrm{C}\text{-}\mathrm{E }~\mathcal{W}$
complex $V$.
\end{df}

There are a few useful adjoint relationships between the category of $R$-modules and the category
of $R$-complexes.

\begin{lem} \label{lem 3.5} (\cite[Lemma 3.1]{Gil04})
For any $R$-module $M$ and any $R$-complex $X$, we have the following natural isomorphisms:\\
\indent $\mathrm{(1)}$ $\mathrm{Hom}_{\mathcal{C}(R)}(D^{n}(M), X) \cong \mathrm{Hom}_{R}(M, X_{n})$.\\
\indent $\mathrm{(2)}$ $\mathrm{Hom}_{\mathcal{C}(R)}(S^{n}(M), X) \cong \mathrm{Hom}_{R}(M, \mathrm{Z}_{n}(X))$.\\
\indent $\mathrm{(3)}$ $\mathrm{Hom}_{\mathcal{C}(R)}(X, D^{n}(M)) \cong \mathrm{Hom}_{R}(X_{n-1}, M)$.\\
\indent $\mathrm{(4)}$  $\mathrm{Hom}_{\mathcal{C}(R)}(X, S^{n}(M)) \cong \mathrm{Hom}_{R}(X_{n}/\mathrm{B}_{n}(X), M)$.
\end{lem}

Now we are in a position to prove the following result in Theorem B from the introduction.
In the next, we use both the subscript notation for degrees of complex and the superscript notation to distinguish
complexes: for example, if $(C^{i})_{i\in \mathbb{Z}}$ is a family of complexes, then $C^{i}_{n}$ denotes the degree-$n$
term of the complex $C^{i}$.

\begin{thm}\label{thm 3.6}
For a complex $G$, the following are equivalent:\\
\indent $\mathrm{(1)}$  $G$ is a $\mathrm{C}$-$\mathrm{E}$ $\mathcal{W}$-Gorenstein complex.\\
\indent $\mathrm{(2)}$  $G$ admits a $\mathrm{C}$-$\mathrm{E}$ complete $\mathcal{W}$ resolution.\\
\indent $\mathrm{(3)}$  $\mathrm{B}(G)$ and $\mathrm{H}(G)$ are complexes consisting of $\mathcal{W}$-Gorenstein modules.
\end{thm}

\begin{proof}
(1)$\Longrightarrow$(2) For any $n\in \mathbb{Z}$, consider the exact sequence of modules
 $0\rightarrow \mathrm{B}_{n}(G)\rightarrow \mathrm{Z}_{n}(G)\rightarrow \mathrm{H}_{n}(G)\rightarrow 0$,
 where B$_{n}(G)$ and H$_{n}(G)$ are $\mathcal{W}$-Gorenstein. Suppose $W^{\mathrm{B}_{n}(G)}$ and
 $W^{\mathrm{H}_{n}(G)}$ are complete $\mathcal{W}$-resolutions of B$_{n}(G)$ and H$_{n}(G)$, respectively.
 By the Horseshoe Lemma, we can construct a complete $\mathcal{W}$ resolution of Z$_{n}(G)$:
 $W^{\mathrm{Z}_{n}(G)}= W^{\mathrm{B}_{n}(G)}\oplus W^{\mathrm{H}_{n}(G)}$. Similarly, consider
 the exact sequence of modules
 $0\rightarrow \mathrm{Z}_{n}(G)\rightarrow G_{n}\rightarrow \mathrm{B}_{n-1}(G)\rightarrow 0$,
 and we can construct a complete $\mathcal{W}$-resolution of $G_{n}$:
 $$W^{G_{n}} = W^{\mathrm{Z}_{n}(G)}\oplus W^{\mathrm{B}_{n-1}(G)} =
 W^{\mathrm{B}_{n}(G)}\oplus W^{\mathrm{H}_{n}(G)}\oplus W^{\mathrm{B}_{n-1}(G)}.$$
 Set
 $W^{i}_{n} = W_{i}^{\mathrm{B}_{n}(G)}\oplus W_{i}^{\mathrm{H}_{n}(G)}\oplus W_{i}^{\mathrm{B}_{n-1}(G)}$ and
 $d_{n}^{W^{i}}: W^{i}_{n}\rightarrow W^{i}_{n-1}$ which maps $(x, y, z)$ to  $(z, 0, 0)$ for all $i, n\in \mathbb{Z}$.
 Then $(W^{i}, d^{W^{i}})$ is a complex such that $G_{n} = \mathrm{Ker}(W^{-1}_{n}\rightarrow W^{-2}_{n})$.

Consider the complex of complexes
$$\mathbb{W} = \cdots\longrightarrow W^{1}\longrightarrow W^{0}\longrightarrow  W^{-1}\longrightarrow  W^{-2}
 \longrightarrow \cdots$$
For any $n\in \mathbb{Z}$, $\mathbb{W}_{n} = \cdots\rightarrow W^{1}_{n}\rightarrow W^{0}_{n} \rightarrow W^{-1}_{n} \rightarrow \cdots$
 is a complete $\mathcal{W}$-resolution of $G_{n}$, and $\mathrm{Z}_{n}(\mathbb{W}) = \cdots \rightarrow \mathrm{Z}_{n}(W^{1}) \rightarrow
 \mathrm{Z}_{n}(W^{0}) \rightarrow \mathrm{Z}_{n}(W^{-1}) \rightarrow \cdots$ is a complete $\mathcal{W}$-resolution
 of Z$_{n}(G)$, so they both are exact. Hence, we can get that $\mathbb{W}$
 is C-E exact. It is easily seen that $W^{i}$ is a C-E $\mathcal{W}$ complex for all $i\in \mathbb{Z}$
 and $G = \mathrm{Ker}(W^{-1}\rightarrow W^{-2})$. It remains to prove that $\mathbb{W}$ is still exact when
 Hom$_{\mathcal{C}(R)}(V, -)$ and Hom$_{\mathcal{C}(R)}(-, V)$ applied to it for any C-E $\mathcal{W}$ complex $V$.

 However, it suffices to prove, by Proposition \ref{prop 2.3}, that the assertion holds when we pick $V$ particularly as $V = D^{n}(M)$ and
 $V = S^{n}(M)$ for any module $M\in \mathcal{W}$ and all $n\in \mathbb{Z}$. By Lemma \ref{lem 3.5}, there is a natural isomorphism
 $\mathrm{Hom}_{\mathcal{C}(R)}(D^{n}(M), \mathbb{W})\cong \mathrm{Hom}_{R}(M, \mathbb{W}_{n})$, i.e. there is a commutative diagram
 $$\xymatrix@C=8pt{
 \cdots \ar[r] & \mathrm{Hom}_{C(R)}(D^{n}(M), W^{1}) \ar[d]^{\cong} \ar[r]
 & \mathrm{Hom}_{C(R)}(D^{n}(M), W^{0}) \ar[d]^{\cong} \ar[r]
 & \mathrm{Hom}_{C(R)}(D^{n}(M), W^{-1}) \ar[d]^{\cong} \ar[r] &\cdots\\
\cdots \ar[r] & \mathrm{Hom}_{R}(M, W^{1}_{n}) \ar[r]
 & \mathrm{Hom}_{R}(M, W^{0}_{n}) \ar[r]
 & \mathrm{Hom}_{R}(M, W^{-1}_{n}) \ar[r] &\cdots
 }$$
 Note that $\mathbb{W}_{n}$ is a complete $\mathcal{W}$-resolutions of $G_n$, then the exactness of the upper row follows since the
 bottom row is exact. Similarly, we have exactness of $\mathrm{Hom}_{\mathcal{C}(R)}(\mathbb{W}, D^{n}(M))$ by
 $\mathrm{Hom}_{\mathcal{C}(R)}(\mathbb{W}, D^{n}(M))\cong \mathrm{Hom}_{R}(\mathbb{W}_{n-1}, M)$, and exactness of
 $\mathrm{Hom}_{\mathcal{C}(R)}(S^{n}(M), \mathbb{W})$ by
 $\mathrm{Hom}_{\mathcal{C}(R)}(S^{n}(M), \mathbb{W})\cong \mathrm{Hom}_{R}(M, \mathrm{Z}_{n}(\mathbb{W}))$.

Moreover, it yields from the exact sequence
 $0\rightarrow \mathrm{H}_{n}(W^{i})\rightarrow W^{i}_{n}/\mathrm{B}_{n}(W^{i})\rightarrow \mathrm{B}_{n-1}(W^{i})\rightarrow 0$
that $W^{i}_{n}/\mathrm{B}_{n}(W^{i})\in \mathcal{W}$, and then there exists an exact sequence of complexes
$$0\longrightarrow \mathrm{B}_{n}(\mathbb{W})\longrightarrow \mathbb{W}_{n}\longrightarrow
 \mathbb{W}_{n}/\mathrm{B}_{n}(\mathbb{W})\longrightarrow 0,$$ which is split degreewise.
Noting that $\mathrm{B}_{n}(\mathbb{W})$ is a complete $\mathcal{W}$-resolution
 of B$_{n}(G)$, the complex $\mathrm{Hom}_{R}(\mathrm{B}_{n}(\mathbb{W}), M)$ is exact. From the exact sequence of $\mathbb{Z}$-complexes
 $$0\longrightarrow \mathrm{Hom}_{R}(\mathbb{W}_{n}/\mathrm{B}_{n}(\mathbb{W}), M)\longrightarrow
 \mathrm{Hom}_{R}(\mathbb{W}_{n}, M)\longrightarrow  \mathrm{Hom}_{R}(\mathrm{B}_{n}(\mathbb{W}), M)\longrightarrow 0,$$
it yields that $\mathrm{Hom}_{R}(\mathbb{W}_{n}/\mathrm{B}_{n}(\mathbb{W}), M)$ is exact since the other two items are so.
Hence, $\mathrm{Hom}_{\mathcal{C}(R)}(\mathbb{W}, S^{n}(M))\cong \mathrm{Hom}_{R}(\mathbb{W}_{n}/\mathrm{B}_{n}(\mathbb{W}), M)$ is exact.

 (2)$\Longrightarrow$ (1)  Suppose that
 $$\mathbb{W} = \cdots\longrightarrow W^{1}\longrightarrow W^{0}\longrightarrow  W^{-1}\longrightarrow  W^{-2}\longrightarrow \cdots$$
 is a C-E complete $\mathcal{W}$-resolution such that $G = \mathrm{Ker}\,(W^{-1}\rightarrow W^{-2})$.
Let $M\in\mathcal{W}$. It follows by the exactness of
$\mathrm{Hom}_{\mathcal{C}(R)}(D^{n}(M), \mathbb{W})$ and $\mathrm{Hom}_{\mathcal{C}(R)}(\mathbb{W}, D^{n+1}(M))$ that
$$\mathbb{W}_{n} = \cdots\longrightarrow W^{1}_{n}\longrightarrow W^{0}_{n}\longrightarrow  W^{-1}_{n}\longrightarrow  W^{-2}_{n}\longrightarrow \cdots$$
remains exact after applying $\mathrm{Hom}_{R}(M, -)$ and $\mathrm{Hom}_{R}(-, M)$. Then, $\mathbb{W}_{n}$
is a complete $\mathcal{W}$-resolution of $G_{n} = \mathrm{Ker}(W^{-1}_{n}\rightarrow W^{-2}_{n})$,
and thus $G$ is a complex of $\mathcal{W}$-Gorenstein modules.

We now argue that $\mathrm{Z}(G)$ is also a complex of $\mathcal{W}$-Gorenstein modules. So we must show that
$\mathrm{Z}_{n}(G)$ is $\mathcal{W}$-Gorenstein for any $n$. We have a candidate for a complete $\mathcal{W}$-resolution of
$\mathrm{Z}_{n}(G)$, namely  an exact sequence
$$\mathrm{Z}_{n}(\mathbb{W}) = \cdots \longrightarrow \mathrm{Z}_{n}(W^{1}) \longrightarrow
 \mathrm{Z}_{n}(W^{0}) \longrightarrow \mathrm{Z}_{n}(W^{-1}) \longrightarrow \cdots$$ with
 each module $\mathrm{Z}_{n}(W^{i})\in \mathcal{W}$, such that
 $\mathrm{Z}_{n}(G) = \mathrm{Ker}(\mathrm{Z}_{n}(W^{-1})\rightarrow \mathrm{Z}_{n}(W^{-2}))$.
Clearly, $\mathrm{Hom}_{R}(M, \mathrm{Z}_{n}(\mathbb{W})) \cong \mathrm{Hom}_{\mathcal{C}(R)}(S^{n}(M), \mathbb{W})$ is exact.
Moreover, in the exact sequence
$0\rightarrow \mathrm{Hom}_{R}(\mathbb{W}_{n}/\mathrm{B}_{n}(\mathbb{W}), M)\rightarrow
 \mathrm{Hom}_{R}(\mathbb{W}_{n}, M)\rightarrow  \mathrm{Hom}_{R}(\mathrm{B}_{n}(\mathbb{W}), M)\rightarrow 0$
of complexes,  $\mathrm{Hom}_{R}(\mathbb{W}_{n}, M)$ and $\mathrm{Hom}_{R}(\mathbb{W}_{n}/\mathrm{B}_{n}(\mathbb{W}), M)\cong
\mathrm{Hom}_{\mathcal{C}(R)}(\mathbb{W}, S^{n}(M))$ are exact,
then so is $\mathrm{Hom}_{R}(\mathrm{B}_{n}(\mathbb{W}), M)$.
Now consider the following exact sequence of complexes
 $$0\longrightarrow \mathrm{Z}_{n}(\mathbb{W})\longrightarrow \mathbb{W}_{n}\longrightarrow
 \mathrm{B}_{n-1}(\mathbb{W})\longrightarrow 0,$$
 which is split degreewised and yields an exact sequence of $\mathbb{Z}$-complexes
 $$0\longrightarrow \mathrm{Hom}_{R}(\mathrm{B}_{n-1}(\mathbb{W}), M)\longrightarrow
 \mathrm{Hom}_{R}(\mathbb{W}_{n}, M)\longrightarrow  \mathrm{Hom}_{R}(\mathrm{Z}_{n}(\mathbb{W}), M)\longrightarrow 0.$$
Then $\mathrm{Hom}_{R}(\mathrm{Z}_{n}(\mathbb{W}), M)$ is exact.
This implies that $\mathrm{Z}_{n}(G)$ is $\mathcal{W}$-Gorenstein.

Furthermore, we have from the exact sequence $0\rightarrow \mathrm{Z}_{n+1}(\mathbb{W})\rightarrow \mathbb{W}_{n+1}\rightarrow
 \mathrm{B}_{n}(\mathbb{W})\rightarrow 0$ that $\mathrm{B}_{n}(\mathbb{W})$ is a complete $\mathcal{W}$-resolution of
 $\mathrm{B}_{n}(G)$, and then from $0\rightarrow\mathrm{B}_{n}(\mathbb{W})\rightarrow \mathrm{Z}_{n}(\mathbb{W})\rightarrow
 \mathrm{H}_{n}(\mathbb{W})\rightarrow 0$ that $\mathrm{H}_{n}(\mathbb{W})$ is a complete $\mathcal{W}$-resolution of
$\mathrm{H}_{n}(G)$. Hence, $\mathrm{B}(G)$ and $\mathrm{H}(G)$ are both complexes of $\mathcal{W}$-Gorenstein modules.

(1)$\Longrightarrow$(3) is trivial.

(3)$\Longrightarrow$(1) Since the class $\mathcal{G}(\mathcal{W})$ of $\mathcal{W}$-Gorenstein modules is closed
under extensions \cite[Corollary 4.5]{SWSW08}, the assertion follows from the exact sequences
 $0\rightarrow \mathrm{B}_{n}(G)\rightarrow \mathrm{Z}_{n}(G)\rightarrow \mathrm{H}_{n}(G)\rightarrow 0$
and  $0\rightarrow \mathrm{Z}_{n}(G)\rightarrow G_{n}\rightarrow \mathrm{B}_{n-1}(G)\rightarrow 0$.
\end{proof}

It is an important question to establish relationships between a complex $X$ and the modules $X_{n}, n\in\mathbb{Z}$.
If $R$ is an $n$-Gorenstein ring, Enochs and Garcia Rozas showed in \cite{EGR98} that a complex
$X$ is Gorenstein projective (resp. Gorenstein injective) if and only if $X_{n}$ is a Gorenstein projective
(resp. Gorenstein injective) module for any $n\in\mathbb{Z}$. This has been further developed by Liu and Zhang \cite{LZ09}
and Yang \cite{Y11}, and now we know that the same result holds over any associative ring.

In \cite[Section 5.1]{Liang}, the author introduced the notion of $\widetilde{\mathcal{W}}$-Gorenstein complex,
analogous to the definition of $\mathcal{W}$-Gorenstein module,
by replacing the modules in $\mathcal{W}$ with complexes in $\widetilde{\mathcal{W}}$. It is proved that
(\cite[Sect. 5.1, Theorem A]{Liang}): a complex $X$ is
$\widetilde{\mathcal{W}}$-Gorenstein if and only if $X_{n}$ is a $\mathcal{W}$-Gorenstein module for each $n\in \mathbb{Z}$.

\begin{cor}\label{cor 3.7}
For a complex $G$, the following are equivalent:\\
\indent $(\mathrm{1})$  $G$ is a $\mathrm{C}$-$\mathrm{E}$ $\mathcal{W}$-Gorenstein complex.\\
\indent $(\mathrm{2})$  $G$, $\mathrm{Z}(G)$, $\mathrm{B}(G)$ and $\mathrm{H}(G)$ are $\widetilde{\mathcal{W}}$-Gorenstein complexes.\\
\indent $(\mathrm{3})$  $\mathrm{B}(G)$ and  $\mathrm{H}(G)$ are $\widetilde{\mathcal{W}}$-Gorenstein complexes.
\end{cor}

In particular, if we set $\mathcal{W}$ to be the class of injective modules $\mathcal{I}$, then
\begin{cor}\label{cor 3.8}(\cite[Theorem 8.5]{E11})
For a complex $G$, the following are equivalent:\\
\indent $(\mathrm{1})$  $G$ has a $\mathrm{C}$-$\mathrm{E}$ complete injective resolution.\\
\indent $(\mathrm{2})$  $G$, $\mathrm{Z}(G)$, $\mathrm{B}(G)$ and $\mathrm{H}(G)$ are complexes consisting of Gorenstein injective modules.
\end{cor}

We will not state here, but there are dual results about C-E Gorenstein projective complexes if $\mathcal{W}$ is
the class of projective modules $\mathcal{P}$.
In particular, set $\mathcal{W} = \mathcal{P}_{C}$ and
$\mathcal{W} = \mathcal{I}_{C}$ respectively, where $C$ is a given semi-dualizing module over a commutative notherian ring $R$.
In \cite{GD11}, $\mathcal{P}_{C}$-Gorenstein and $\mathcal{I}_{C}$-Gorenstein modules are named respectively
$C$-Gorenstein projective and $C$-Gorenstein injective modules. Accordingly,
C-E $\mathcal{P}_{C}$-Gorenstein and C-E $\mathcal{I}_{C}$-Gorenstein complexes are called C-E $C$-Gorenstein projective
and C-E $C$-Gorenstein injective complexes respectively.

\section{\bf Foxby equivalence}
This section is devoted to prove Theorem A from the introduction. The proof is divided into the following results.

Throughout this section, $R$ is a noetherian ring, and $C$ is a given semi-dualizing module over $R$.
Foxby \cite{Fox72} studied modules in Auslander class $\mathcal{A}_{C}^{0}(R)$ and Bass class $\mathcal{B}_{C}^{0}(R)$, where
$$\mathcal{A}_{C}^{0}(R)= \left\{M \in R\text{-}\mathrm{Mod}\, \vline \begin{matrix}\, \textrm{Tor}^{R}_{i}(C, M) = 0 = \mathrm{Ext}^{i}_{R}(C, C\otimes_{R}M),
\text{ and the }\\ \, \textrm{ map }
 M\longrightarrow
 \mathrm{Hom}_{R}(C,  C\otimes_{R}M) \text{ is an isomorphism}
 \end{matrix} \right\},
$$
and
$$\mathcal{B}_{C}^{0}(R)= \left\{N \in R\text{-}\mathrm{Mod}\, \vline \begin{matrix}\, \mathrm{Ext}_{R}^{i}(C, N) = 0 = \textrm{Tor}_{i}^{R}(C, \mathrm{Hom}_{R}(C, N)),
\text{ and the }\\ \, \textrm{ map }
 C\otimes_{R}\mathrm{Hom}_{R}(C,  N)\longrightarrow N
  \text{ is an isomorphism}
 \end{matrix} \right\}.
$$
We always take $C$ as a complex concentrated in degree zero, and then it is a semi-dualizing
complex in the sense of \cite[Definition 2.1]{Chr01}.
By \cite[Observation 4.10]{Chr01}, $\mathcal{A}_{C}^{0}(R)$ and $\mathcal{B}_{C}^{0}(R)$
coincide with the subcategories
of $\mathcal{A}_{C}(R)$ and $\mathcal{B}_{C}(R)$ consisting of $R$-complexes concentrated in degree zero, respectively.

\begin{df}\label{df 4.1}
Let $C$ be a semi-dualizing module over a noetherian ring $R$.
The $\mathrm{C}$-$\mathrm{E}$ Auslander class and $\mathrm{C}$-$\mathrm{E}$ Bass class with respect to $C$,
denoted by $\mathrm{CE}$-$\mathcal{A}_{C}(R)$ and $\mathrm{CE}$-$\mathcal{B}_{C}(R)$, are the full subcategories of $\mathcal{D}_{b}(R)$
whose objects are specified as follows:
$$\mathrm{CE}\text{-}\mathcal{A}_{C}(R)= \left\{X\in \mathcal{D}_{b}(R) \, \vline \,\begin{matrix}  X\simeq A, \text{where }
A, \mathrm{Z}(A), \mathrm{B}(A) \text{ and }\mathrm{H}(A) \text{ are}\\
\text{complexes consisting of modules in } \mathcal{A}_{C}^{0}(R)
\end{matrix} \right\}
$$
and
$$\mathrm{CE}\text{-}\mathcal{B}_{C}(R)= \left\{X\in \mathcal{D}_{b}(R) \, \vline \,\begin{matrix}  X\simeq D, \text{where }
D, \mathrm{Z}(D), \mathrm{B}(D) \text{ and }\mathrm{H}(D) \text{ are}\\
\text{complexes consisting of modules in } \mathcal{B}_{C}^{0}(R)
\end{matrix} \right\}.$$
\end{df}

\begin{prop}\label{prop 4.2}
Let $R$ be a noetherian ring with a semi-dualizing module $C$, $X$ an $R$-complex.
If $X\in \mathrm{CE}\text{-}\mathcal{A}_{C}(R)$, then the complexes $X$, $\mathrm{Z}(X)$, $\mathrm{B}(X)$ and $\mathrm{H}(X)$
are all in $\mathcal{A}_{C}(R)$. Dually, if $X\in \mathrm{CE}\text{-}\mathcal{B}_{C}(R)$, then the complexes
$X$, $\mathrm{Z}(X)$, $\mathrm{B}(X)$ and $\mathrm{H}(X)$
are all in $\mathcal{B}_{C}(R)$.
\end{prop}

\begin{proof}
Let $X\in\mathrm{CE}\text{-}\mathcal{A}_{C}(R)$. Then there exists an isomorphism $X\simeq A$ in
$\mathcal{D}(R)$, where $A$ is a complex such that $A$, $\mathrm{Z}(A)$, $\mathrm{B}(A)$ and $\mathrm{H}(A)$ are
complexes consisting of modules in $\mathcal{A}_{C}^{0}(R)$. Set $\mathrm{sup}(X)=s$, $\mathrm{inf}(X) = i$.
Consider a truncated complex
$$A_{(s, i)} = 0\longrightarrow A_{s}/\mathrm{B}_{s}(A)\longrightarrow  A_{s-1}\longrightarrow  \cdots \longrightarrow A_{i+1}\longrightarrow \mathrm{Z}_{i}(A)\longrightarrow 0.$$
By noting that in the exact sequence $0\rightarrow \mathrm{B}_{s}(A)\rightarrow A_{s}\rightarrow A_{s}/\mathrm{B}_{s}(A)\rightarrow 0$
the first two entries are in $\mathcal{A}_{C}^{0}(R)$, the third is also in $\mathcal{A}_{C}^{0}(R)$, and then we have
$A\simeq A_{(s, i)}\in \mathrm{CE}\text{-}\mathcal{A}_{C}(R)$. Thus, we may choose $A$ to be a bounded complex consisting of modules in
$\mathcal{A}_{C}^{0}(R)$.

Let $\alpha: P^{\bullet}\rightarrow C$
be a projective resolution of the semi-dualizing module $C$, where
$P^{\bullet} = \cdots \rightarrow P_{1}\rightarrow P_{0}\rightarrow 0$.
Then $\alpha$ is a quasi-isomorphism, and we represent
$C\otimes^{\mathbf{L}}_{R}X \simeq C\otimes^{\mathbf{L}}_{R}A$ by the complex $P^{\bullet}\otimes_{R}A$. For any $n\in \mathbb{Z}$,
it follows from Tor$^{i}_{R}(C, A_{n}) = 0$ that $\alpha\otimes_{R}A_{n}: P^{\bullet}\otimes_{R}A_{n} \rightarrow C\otimes_{R}A_{n}$ is a
quasi-isomorphism. By \cite[Proposition 2.14]{CFH06},
$\alpha \otimes _{R}A: P^{\bullet}\otimes_{R}A \rightarrow C\otimes_{R}A$ is then a quasi-isomorphism, and hence
$C\otimes^{\mathbf{L}}_{R}A\simeq P^{\bullet}\otimes_{R}A\simeq C\otimes_{R}A\in \mathcal{D}_{b}(R)$.

It follows that $\mathrm{Hom}_{R}(C, C\otimes_{R}A_{n}) \rightarrow \mathrm{Hom}_{R}(P^{\bullet}, C\otimes_{R}A_{n})$ is a
quasi-isomorphism since $\mathrm{Ext}^{i}_{R}(C, C\otimes A_{n}) = 0$. Then, we have a quasi-isomorphism
$\mathrm{Hom}_{R}(\alpha, C\otimes _{R}A): \mathrm{Hom}_{R}(C, C\otimes_{R}A) \rightarrow \mathrm{Hom}_{R}(P^{\bullet}, C\otimes_{R}A)$
by \cite[Proposition 2.7]{CFH06}. Hence
$\mathbf{R}\mathrm{Hom}_{R}(C,  C\otimes_{R}^{\mathbf{L}}A)\simeq \mathrm{Hom}_{R}(P^{\bullet},  C\otimes_{R}A)\simeq \mathrm{Hom}_{R}(C,  C\otimes_{R}A)$.
Moreover, $A\rightarrow\mathrm{Hom}_{R}(C,  C\otimes_{R}A)$ is a canonical isomorphism. Then
$A\rightarrow \mathbf{R}\mathrm{Hom}_{R}(C,  C\otimes_{R}^{\mathbf{L}}A)$ is an isomorphism in $\mathcal{D}(R)$, and this implies $X\simeq A\in \mathcal{A}_{C}(R)$.
Similarly, the bounded complex  $\mathrm{H}(A)$ is also in $\mathcal{A}_{C}(R)$, and so
$\mathrm{H}(X)\simeq \mathrm{H}(A)\in \mathcal{A}_{C}(R)$. In the next, we need to prove that the complexes $\mathrm{Z}(X)$ and $\mathrm{B}(X)$
are in $\mathcal{A}_{C}(R)$.

Conversely, we suppose $\mathrm{Z}(X)$ is not in $\mathcal{A}_{C}(R)$. Note that $\mathcal{A}_{C}(R)$ is a triangulated subcategory of $\mathcal{D}_{b}(R)$,
which satisfies 2-out-of-3 property, that is, if any two items in an exact sequence are in $\mathcal{A}_{C}(R)$ then so is the third.
Then the exact sequence $0\rightarrow \mathrm{B}(X)\rightarrow \mathrm{Z}(X)\rightarrow \mathrm{H}(X)\rightarrow 0$ yields that
$\mathrm{B}(X)\not\in\mathcal{A}_{C}(R)$ (otherwise, contradict to $\mathrm{Z}(X)\not\in\mathcal{A}_{C}(R)$). Moreover, there is an exact sequence
$0\rightarrow \mathrm{Z}(X)\rightarrow X\rightarrow \mathrm{B}(X)[1]\rightarrow 0$, where $\mathrm{B}(X)[1]$ is the complex obtained by
shifting $\mathrm{B}(X)$ one-degree to the left, and it yields that $X\not\in\mathcal{A}_{C}(R)$. Hence, a contradiction occurs, which implies that
both $\mathrm{Z}(X)$ and $\mathrm{B}(X)$ are in $\mathcal{A}_{C}(R)$.

The rest of the assertions can be proved dually, and then is omitted.
\end{proof}

\begin{prop}\label{prop 4.3}
There is an equivalence of categories:
$$\xymatrix@C=100pt {
 \mathrm{CE}\text{-}\mathcal{A}_{C}(R) \ar@<0.5ex>[r]^{C\otimes_{\mathbf{R}} ^{\mathbf{L}}-}
& \mathrm{CE}\text{-}\mathcal{B}_{C}(R)\ar@<0.5ex>[l]^{\mathbf{R}\mathrm{Hom}_{R}(C, -)}
.}$$
\end{prop}

\begin{proof}
Let $X\in\mathrm{CE}\text{-}\mathcal{A}_{C}(R)$. Then there exists an isomorphism $X\simeq A$ in
$\mathcal{D}(R)$, where $A$ is a bounded complex such that $A$, $\mathrm{Z}(A)$, $\mathrm{B}(A)$ and $\mathrm{H}(A)$ are
complexes consisting of modules in $\mathcal{A}_{C}^{0}(R)$.

It follows from the arguments above that
$C\otimes^{\mathbf{L}}_{R}X\simeq C\otimes_{R}A$.
Obviously, $C\otimes_{R}A$ is a complex of modules in $\mathcal{B}_{C}^{0}(R)$.
It remains to prove that Z$(C\otimes_{R}A)$, B$(C\otimes_{R}A)$ and H$(C\otimes_{R}A)$ are all complexes
of modules in $\mathcal{B}_{C}^{0}(R)$.

For any $n\in \mathbb{Z}$, consider the exact sequence $0\rightarrow\mathrm{Z}_{n}(A)\rightarrow A_{n}
\stackrel{d_{n}^{A}}\rightarrow  \mathrm{B}_{n-1}(A)\rightarrow 0$.
Since $\mathrm{B}_{n-1}(A)\in \mathcal{A}_{C}^{0}(R)$, then there is an exact sequence
$$0\longrightarrow C\otimes_{R}\mathrm{Z}_{n}(A)\longrightarrow C\otimes_{R}A_{n}
\stackrel{C\otimes_{R}d_{n}^{A}}\longrightarrow  C\otimes_{R}\mathrm{B}_{n-1}(A)\longrightarrow 0.$$
So $\mathrm{Z}_{n}(C\otimes_{R}P) = C\otimes_{R}\mathrm{Z}_{n}(P)\in \mathcal{B}_{C}^{0}(R)$ and
$B_{n-1}(C\otimes_{R}P) = C\otimes_{R}\mathrm{B}_{n-1}(P)\in \mathcal{B}_{C}^{0}(R)$.
Similarly, consider the exact sequence $0\rightarrow \mathrm{B}_{n}(A)\rightarrow \mathrm{Z}_{n}(A)\rightarrow \mathrm{H}_{n}(A)\rightarrow 0$,
and we have, from the exact sequence
$$0\longrightarrow C\otimes_{R}\mathrm{B}_{n}(A)\longrightarrow C\otimes_{R}\mathrm{Z}_{n}(A)
\longrightarrow  C\otimes_{R}\mathrm{H}_{n}(A)\longrightarrow 0,$$
that $\mathrm{H}_{n}(C\otimes_{R}A) = C\otimes_{R}\mathrm{H}_{n}(A)\in \mathcal{B}_{C}^{0}(R)$.

The proof that $\mathbf{R}\mathrm{Hom}_{R}(C, -)$ takes $\mathrm{CE}\text{-}\mathcal{B}_{C}(R)$ into
$\mathrm{CE}\text{-}\mathcal{A}_{C}(R)$ is similar.
Finally, it follows from Proposition \ref{prop 4.2} that there are inclusions of categories
$\mathrm{CE}\text{-}\mathcal{A}_{C}(R)\subseteq \mathcal{A}_{C}(R)$ and
$\mathrm{CE}\text{-}\mathcal{B}_{C}(R)\subseteq \mathcal{B}_{C}(R)$,
and hence the equivalence of categories is immediate by \cite[Theorem 4.6]{Chr01}.
This completes the proof.
\end{proof}

The subcategory
$\{X\in \mathcal{D}_b(R) | X\simeq P \text{ with } P \text{ a bounded C-E projective complex} \}$ of $\mathcal{D}_b(R)$
is denoted by CE-$\overline{\mathcal{P}}(R)$.
Similarly, CE-$\overline{\mathcal{P}}_{C}(R)$ denotes the subcategory of $\mathcal{D}_b(R)$ consisting of
complexes isomorphic to bounded C-E $C$-projective complexes.

By \cite[Lemmas 4.1, 5.1]{HW07}, the Auslander class $\mathcal{A}_{C}^{0}(R)$
contains every flat (projective) $R$-module and every $C$-injective $R$-module, and the Bass class $\mathcal{B}_{C}^{0}(R)$ contains every
 injective $R$-module and every $C$-projective $R$-module. We have:

\begin{prop}\label{prop 4.4}
\indent $(1)$ There are inclusions of categories:
\begin{center}
 $\mathrm{CE}\text{-}\overline{\mathcal{P}}(R) \subseteq \overline{\mathcal{P}}(R)\cap \mathrm{CE}\text{-}\mathcal{A}_{C}(R)$,
 $\mathrm{CE}\text{-}\overline{\mathcal{P}_{C}}(R) \subseteq \overline{\mathcal{P}_{C}}(R)\cap \mathrm{CE}\text{-}\mathcal{B}_{C}(R)$.\end{center}
 \indent $(2)$ There is an equivalence of categories:
\begin{center}$\xymatrix@C=100pt{
 \mathrm{CE}\text{-}\overline{\mathcal{P}}(R) \ar@<0.5ex>[r]^{C\otimes_{\mathbf{R}}^{\mathbf{L}}-}
& \mathrm{CE}\text{-}\overline{\mathcal{P}}_{C}(R)\ar@<0.5ex>[l]^{\mathbf{R}\mathrm{Hom}_{R}(C, -)}}$.\end{center}
\end{prop}

\begin{proof}
The inclusions of categories are obvious. We only need to prove (2).

Let $X\in \mathrm{CE}\text{-}\overline{\mathcal{P}}(R)$. Then there exists an isomorphism $X\simeq P$ in
$\mathcal{D}(R)$, where $P$ is a bounded C-E projective complex. By Corollary \ref{cor 2.4}, we may choose
$P$ to be a graded module of projectives. Note that $C\otimes_{\mathbf{R}}^{\mathbf{L}}X\simeq C\otimes_{R}P$
and $C\otimes_{R}P$ is a graded module of with items being $C$-projective modules. Then $C\otimes_{\mathbf{R}}^{\mathbf{L}}X$
is in $\mathrm{CE}\text{-}\overline{\mathcal{P}}_{C}(R)$.

For any $X\in\mathrm{CE}\text{-}\overline{{P}_{C}}(R)$, it is proved similarly that
$\mathbf{R}\mathrm{Hom}_{R}(C, X)\in \mathrm{CE}\text{-}\overline{P}(R)$.
Moreover, by Proposition \ref{prop 4.3} the equivalence of categories holds.
\end{proof}

We denote by $\mathcal{G}(\mathcal{P})$, $\mathcal{G}(\mathcal{I})$, $\mathcal{G}(\mathcal{P}_{C})$ and $\mathcal{G}(\mathcal{I}_{C})$ the class of
Gorenstein projective, Gorenstein injective, $C$-Gorenstein projective and $C$-Gorenstein injective modules respectively.

Recall from \cite{HJ06} and \cite{W10} that a complete $\mathcal{P}\mathcal{P}_{C}$-resolution is a complex
$X$ of $R$-modules satisfying: (1) The complex $X$ is exact and $\mathrm{Hom}_{R}(-, \mathcal{P}_{C})$-exact; (2) The $R$-module $X_i$
is projective if $i\geq 0$ and $X_i$ is $C$-projective if $i < 0$. An $R$-module $M$ is $\mathcal{G}_{C}$-projective if
there exists a complete $\mathcal{P}\mathcal{P}_{C}$-resolution  $X$ such that $M\cong \mathrm{Z}_{-1}(X)$. Dually,
a complete $\mathcal{I}_{C}\mathcal{I}$-coresolution and a $\mathcal{G}_{C}$-injective module are defined.
The classes of $\mathcal{G}_{C}$-projective and $\mathcal{G}_{C}$-injective modules are denoted by
$\mathcal{GP}_{C}$ and $\mathcal{GI}_{C}$ respectively.

In the following lemma, (1) is from \cite[Proposition 5.2]{SWSW08} or \cite[Proposition 3.6]{GD11},
and (2) is from \cite[Theorem 3.11]{GD11}.

\begin{lem}\label{lem 4.5}
\indent $(1)$ $\mathcal{G}(\mathcal{P}_{C}) = \mathcal{GP}_{C} \cap \mathcal{B}_{C}^{0}(R)$ and
$\mathcal{G}(\mathcal{I}_{C}) = \mathcal{GI}_{C} \cap \mathcal{A}_{C}^{0}(R)$.\\
\indent $(2)$ Let $M$ be an $R$-module. If $M\in \mathcal{G}(\mathcal{P})\cap \mathcal{A}^{0}_{C}(R)$,
then $C\otimes_{R}M \in \mathcal{G}(\mathcal{P}_{C})$; if $M\in \mathcal{G}(\mathcal{P}_{C})$,
then $\mathrm{Hom}_{R}(C, M)\in \mathcal{G}(\mathcal{P})\cap \mathcal{A}^{0}_{C}(R)$. Dually,
if $M\in \mathcal{G}(\mathcal{I}_{C})$,
then $C\otimes_{R}M \in \mathcal{G}(\mathcal{I})\cap \mathcal{B}^{0}_{C}(R)$;
if $M\in \mathcal{G}(\mathcal{I})\cap \mathcal{B}^{0}_{C}(R)$,
then $\mathrm{Hom}_{R}(C, M)\in \mathcal{G}(\mathcal{I}_{C})$.
\end{lem}

In the next, we consider the subcategory $\{X\in \mathcal{D}_b(R) | X\simeq G \}$ of $\mathcal{D}_b(R)$.
If $G$ is a bounded C-E Gorenstein projective complex, then it is denoted by CE-$\overline{\mathcal{G}(\mathcal{P})}(R)$;
if $G$ is a bounded C-E $C$-Gorenstein projective complex, then we use CE-$\overline{\mathcal{G}(\mathcal{P}_{C})}(R)$
to indicate it.

\begin{prop}\label{prop 4.6}
\indent $(1)$ There are inclusions of categories:
\begin{center}
 $\mathrm{CE}\text{-}\overline{\mathcal{P}}(R) \subseteq \mathrm{CE}\text{-}\overline{\mathcal{G}(\mathcal{P})}(R)\cap \mathrm{CE}\text{-}\mathcal{A}_{C}(R)$,
 $\mathrm{CE}\text{-}\overline{\mathcal{P}_{C}}(R) \subseteq \mathrm{CE}\text{-}\overline{\mathcal{G}(\mathcal{P}_{C}})(R)
 \subseteq \mathrm{CE}\text{-}\mathcal{B}_{C}(R)$.\end{center}
 \indent $(2)$ There is an equivalence of categories:
$$\xymatrix@C=100pt{
 \mathrm{CE}\text{-}\overline{\mathcal{G}(\mathcal{P})}(R)\cap \mathrm{CE}\text{-}\mathcal{A}_{C}(R) \ar@<0.5ex>[r]^{C\otimes_{\mathbf{R}}^{\mathbf{L}}-}
& \mathrm{CE}\text{-}\overline{\mathcal{G}(\mathcal{P}_C)}(R)\ar@<0.5ex>[l]^{\mathbf{R}\mathrm{Hom}_{R}(C, -)}}.$$
\end{prop}

\begin{proof}
Let $X$ be a complex in $\mathcal{D}_b(R)$.
If $X\in\mathrm{CE}\text{-}\overline{\mathcal{G}(\mathcal{P})}(R)\cap \mathrm{CE}\text{-}\mathcal{A}_{C}(R)$,
then $X\simeq G$, where $G$ is a bounded C-E Gorenstein projective complex and is in CE-$\mathcal{A}_C(R)$.
By the argument in Proposition \ref{prop 4.2}, $C\otimes_{\mathbf{R}}^{\mathbf{L}}X\simeq C\otimes_{R}G$.

For $X\in\mathrm{CE}\text{-}\overline{\mathcal{G}(\mathcal{P}_C)}(R)$, it is proved similarly
that $X\simeq G$ where $G$ is a bounded C-E $C$-Gorenstein projective complex, and moreover,
$\mathbf{R}\mathrm{Hom}_{R}(C, X)\simeq \mathrm{Hom}_{R}(C, G)$.
Then, the assertions follow by Lemma \ref{lem 4.5} and Proposition \ref{prop 4.3}.
\end{proof}

\begin{rem}\label{rem 4.7}
The dual  results hold, and then we have the following diagram:
$$\xymatrix@C=100pt{
\mathcal{A}_{C}(R)  \ar@<0.4ex>[r]^{C\otimes^{\mathbf{L}}_{R}-}
&\mathcal{B}_{C}(R) \ar@<0.4ex>[l]^{\mathbf{R}\mathrm{Hom}_{R}(C, -)}\\
\mathrm{CE}\text{-}\mathcal{A}_{C}(R) \ar@{^{(}->}[u] \ar@<0.4ex>[r]^{C\otimes^{\mathbf{L}}_{R}-}
&\mathrm{CE}\text{-}\mathcal{B}_{C}(R) \ar@{^{(}->}[u]\ar@<0.4ex>[l]^{\mathbf{R}\mathrm{Hom}_{R}(C, -)}\\
\mathrm{CE}\text{-}\overline{\mathcal{G}(\mathcal{I}_C)}(R) \ar@{^{(}->}[u]\ar@<0.4ex>[r]^{C\otimes^{\mathbf{L}}_{R}-}
&\mathrm{CE}\text{-}\overline{\mathcal{G}(\mathcal{I})}(R)\cap \mathrm{CE}\text{-}\mathcal{B}_{C}(R)  \ar@{^{(}->}[u]\ar@<0.4ex>[l]^{\mathbf{R}\mathrm{Hom}_{R}(C, -)}\\
\mathrm{CE}\text{-}\overline{\mathcal{I}_{C}}(R) \ar@{^{(}->}[u]\ar@<0.4ex>[r]^{C\otimes^{\mathbf{L}}_{R}-}
&\mathrm{CE}\text{-}\overline{\mathcal{I}}(R) \ar@{^{(}->}[u] \ar@<0.4ex>[l]^{\mathbf{R}\mathrm{Hom}_{R}(C, -)}
}
$$
\end{rem}

\section{\bf Cartan-Eilenberg homological dimensions}
In this section, we are devoted to study homological dimensions of complexes with respect to C-E complexes.
The symbol $\mathcal{D}_{\sqsupset}(R)$ (resp. $\mathcal{D}_{\sqsubset}(R)$)  stands for the full subcategory of
$\mathcal{D}(R)$ consisting of homologically right-bounded (resp.  homologically left-bounded) complexes.

\begin{df}\label{def 5.1}
Let $X$ be a complex in $\mathcal{D}_{\sqsupset}(R)$. Consider the invariant
$$\mathrm{inf}\{\mathrm{sup}\{i\in \mathbb{Z}\mid P_{i} \neq 0 \} \mid X \simeq P \text{ with } P\in \mathcal{C}_{\mathcal{X}}(R)\};$$
if $\mathcal{C}_{\mathcal{X}}(R)$ is the category of $\mathrm{C}$-$\mathrm{E}$ projective complexes,
then it is said to be $\mathrm{C}$-$\mathrm{E}$ projective dimension of $X$,
and is denoted by $\mathrm{CE}\text{-}\mathrm{pd}_{R}X$;
if $\mathcal{C}_{\mathcal{X}}(R)$ is the category of $\mathrm{C}$-$\mathrm{E}$ Gorenstein projective complexes,
then it is said to be $\mathrm{C}$-$\mathrm{E}$ Gorenstein projective dimension of $X$, and is denoted by
$\mathrm{CE}\text{-}\mathrm{Gpd}_{R}X$.
\end{df}

Next, we will show that under reasonable conditions, the set taken in the above definition is not empty.

Recall that a complex $P$ is DG-projective provided that $P_{n}$ is a projective module for any $n\in \mathbb{Z}$ such that every chain map
from $P$ to an exact complex $E$ is homotopic to zero; which is also termed semi-projective complex.
A complex is projective if and only if it is both exact and DG-projective. By \cite[section 10]{E11}, every
projective complex is C-E projective and every C-E projective complex is DG-projective.

\begin{lem}\label{lem 5.2}
Let $X$ be a DG-projective complex with $\mathrm{H}(X)$ consisting of projective modules. Then there exists a split exact sequence
$0\rightarrow K\rightarrow P\rightarrow X\rightarrow 0$, where $P$ is $\mathrm{C}$-$\mathrm{E}$ projective and $K$ is projective.
Consequently, a complex $P$ is $\mathrm{C}$-$\mathrm{E}$ projective if and only if $P$ is DG-projective
with $\mathrm{H}(P)$ a complex of projective modules.
\end{lem}

\begin{proof}
For any $n\in \mathbb{Z}$, consider the exact sequence
$0\rightarrow L_{n}\rightarrow Q_{n}\rightarrow \mathrm{B}_{n}(X)\rightarrow 0$ with $Q_{n}$ projective. By using the Horseshoe Lemma, we have
the following commutative diagrams:
$$\xymatrix@C=25pt {
&0 \ar[d] &0 \ar[d] & 0 \ar[d] \\
0\ar[r] &L_{n} \ar[r]\ar[d] &L_{n} \ar[r]\ar[d] &0 \ar[r]\ar[d] &0\\
0\ar[r] &Q_{n} \ar[r]\ar[d] &Q_{n}\oplus\mathrm{H}_{n}(X) \ar[r]\ar[d] &\mathrm{H}_{n}(X) \ar[r]\ar[d] & 0\\
0\ar[r] &\mathrm{B}_{n}(X) \ar[r]\ar[d] &\mathrm{Z}_{n}(X) \ar[r]\ar[d] &\mathrm{H}_{n}(X) \ar[r]\ar[d]  & 0\\
& 0 & 0 & 0
}$$
and
$$\xymatrix@C=12pt {
&0 \ar[d] &0 \ar[d] & 0 \ar[d] \\
0\ar[r] &L_{n} \ar[r]\ar[d] &K_{n} \ar[r]\ar[d] &L_{n-1} \ar[r]\ar[d] &0\\
0\ar[r] &Q_{n}\oplus\mathrm{H}_{n}(X) \ar[r]\ar[d] &Q_{n}\oplus\mathrm{H}_{n}(X)\oplus Q_{n-1} \ar[r]\ar[d] &Q_{n-1} \ar[r]\ar[d] & 0\\
0\ar[r] &\mathrm{Z}_{n}(X) \ar[r]\ar[d] &X_{n} \ar[r]\ar[d] &\mathrm{B}_{n-1}(X) \ar[r]\ar[d]  & 0\\
& 0 & 0 & 0
}$$
Let $P_{n} = Q_{n}\oplus\mathrm{H}_{n}(X)\oplus Q_{n-1}$. Then $P = (P_{n}, d_{n})$ is a complex with the differential
$d_{n}: P_{n}\rightarrow P_{n-1}$ which is given by $d_{n}(x, y, z) = (z, 0, 0)$. It is direct to check that $P$ is a C-E
projective complex. Moreover, we have an exact complex $K$ by pasting the exact sequences
$0\rightarrow L_{n}\rightarrow K_{n}\rightarrow L_{n-1}\rightarrow 0$ together; $K$ is also DG-projective since both
$X$ and $P$ are so, and then $K$ is projective. Note that the exact sequence
$0\rightarrow K\rightarrow P\rightarrow X\rightarrow 0$ is split degreewise and  $K$ is contractible, hence it is split.
This implies that $X$ is a direct summand of $P$, and then is DG-projective.
\end{proof}

\begin{prop}\label{prop 5.3}
Let $X$ be a complex. Then $X\simeq P$ for a $\mathrm{C}$-$\mathrm{E}$ projective complex $P$ if and only if
$\mathrm{H}(X)$ is a complex of projective modules. In this case, one can chose $P$ to be DG-projective.
\end{prop}

 \begin{proof}
It is well known that for any complex $X$, there exist a (epic) quasi-isomorphism $P\rightarrow X$
with $P$ DG-projective. Then the assertion follows immediately by Lemma \ref{lem 5.2}.
 \end{proof}

\begin{rem}\label{rem 5.4}
\indent$(1)$  For any complex $X\in\mathcal{D}_{\sqsupset}(R)$, it follows from the definition that
$\mathrm{sup}(X)\leq \mathrm{CE}\text{-}\mathrm{Gpd}_{R}X \leq \mathrm{CE}\text{-}\mathrm{pd}_{R}X$.
If $X\simeq 0$, i.e. $X$ is exact, then $\mathrm{CE}\text{-}\mathrm{pd}_{R}X = -\infty = \mathrm{CE}\text{-}\mathrm{Gpd}_{R}X$.\\
\indent(2) Recall that when $\mathcal{C}_{\mathcal{X}}(R)$ in the above definition is the category of right-bounded complexes
of projective modules, then it is projective dimension of $X$, denoted by $\mathrm{pd}_{R}X$;
when $\mathcal{C}_{\mathcal{X}}(R)$ is the category of right-bounded complexes
of Gorenstein projective modules, then it is Gorenstein projective dimension in the
sense of Christensen et. al \cite{Chr00, CFH06}, and is denoted by $\mathrm{Gpd}_{R}X$.
Moreover, we have $\mathrm{pd}_{R}X\leq \mathrm{CE}\text{-}\mathrm{pd}_{R}X$, $\mathrm{Gpd}_{R}X \leq \mathrm{CE}\text{-}\mathrm{Gpd}_{R}X$.
 \end{rem}

\begin{thm}\label{thm 5.5}
Let $X\in\mathcal{D}_{\sqsupset}(R)$. Assume that $\mathrm{H}(X)$ is a complex of projective modules.
Then the following are equivalent:\\
\indent$(1)$ $\mathrm{CE}\text{-}\mathrm{pd}_{R}X$ is finite.\\
\indent$(2)$ $\mathrm{pd}_{R}X$, $\mathrm{pd}_{R}\mathrm{Z}(X)$, $\mathrm{pd}_{R}\mathrm{B}(X)$ and $\mathrm{pd}_{R}\mathrm{H}(X)$ are all finite.
\end{thm}

\begin{proof}
(1)$\Longrightarrow$(2)  Since $\mathrm{CE}\text{-}\mathrm{pd}_{R}X < \infty$, it is trivial that $\mathrm{pd}_{R}X<\infty$. Noting that
$X\simeq P$ for a bounded C-E projective complex $P$, $\mathrm{H}(X)$ is isomorphic to the bounded complex of projective modules $\mathrm{H}(P)$,
and then $\mathrm{pd}_{R}\mathrm{H}(X)<\infty$.

By \cite[1.4.3]{Vel06}, in an exact sequence of complexes, if two of them have finite projective dimension, so does
the third. Consider exact sequences $0\rightarrow \mathrm{B}(X)\rightarrow \mathrm{Z}(X)\rightarrow \mathrm{H}(X)\rightarrow 0$
and
$0\rightarrow \mathrm{Z}(X)\rightarrow X\rightarrow \mathrm{B}(X)[1]\rightarrow 0$, where $\mathrm{B}(X)[1]$ is the complex obtained by
shifting $\mathrm{B}(X)$ one-degree to the left. If either $\mathrm{pd}_{R}\mathrm{Z}(X)=\infty$ or $\mathrm{pd}_{R}\mathrm{B}(X)=\infty$,
a contradiction will occur. It yields that both $\mathrm{pd}_{R}\mathrm{Z}(X)$ and $\mathrm{pd}_{R}\mathrm{B}(X)$ are finite.

(2)$\Longrightarrow$(1) We set $s= \mathrm{pd}_{R}X$ and $i=\mathrm{inf}(X)$.
Then $X\simeq P\simeq P_{(s, i)}$, where $P$ is C-E projective and
$$P_{(s, i)} = 0\longrightarrow P_{s}/\mathrm{B}_{s}(P)\longrightarrow  P_{s-1}\longrightarrow  \cdots \longrightarrow P_{i+1}\longrightarrow \mathrm{Z}_{i}(P)\longrightarrow 0$$
is the truncated complex of $P$. By \cite[Theorem 2.4.P.]{AF91}, $P_{s}/\mathrm{B}_{s}(P)$ is a projective module, and then $P_{(s, i)}$
is also a C-E projective complex. Hence $\mathrm{CE}\text{-}\mathrm{pd}_{R}X \leq s$.
\end{proof}

\begin{cor}\label{cor 5.6}
Let $X\in\mathcal{D}_{\sqsupset}(R)$ be a complex of finite $\mathrm{C}$-$\mathrm{E}$ projective dimension. Then
$\mathrm{CE}\text{-}\mathrm{pd}_{R}X = \mathrm{pd}_{R}X$.
\end{cor}

Modules with excellent duality properties have turned out to be a powerful tool \cite{Har66}.
Recall that a semi-dualizing module $C$ is dualizing provided that
$C$ has finite injective dimension. For C-E Gorenstein projective dimension,
we have the following results from Theorem C in introduction.

\begin{thm}\label{thm 5.7}
Let $X\in\mathcal{D}_{\sqsupset}(R)$.
Assume that $X\simeq G$ for a $\mathrm{C}$-$\mathrm{E}$ Gorenstein projective complex $G$.
Then the following are equivalent:\\
\indent$(1)$  $\mathrm{CE}\text{-}\mathrm{Gpd}_{R}X$ is finite.\\
\indent$(2)$   $\mathrm{Gpd}_{R}X$, $\mathrm{Gpd}_{R}\mathrm{Z}(X)$, $\mathrm{Gpd}_{R}\mathrm{B}(X)$ and $\mathrm{Gpd}_{R}\mathrm{H}(X)$ are all finite.

Moreover, if $R$ is a noetherian ring with a dualizing module $C$, then the above are equivalent to:\\
\indent$(3)$  $X\in \mathrm{CE}\text{-}\mathcal{A}_{C}(R)$.\\
\indent$(4)$  $X$, $\mathrm{Z}(X)$, $\mathrm{B}(X)$ and $\mathrm{H}(X)$ are in $\mathcal{A}_{C}(R)$.
\end{thm}

\begin{proof}
By \cite[Theorem 3.9 (1)]{Vel06}, in an exact sequence of complexes, if two of them have finite Gorenstein projective dimension, then so
does the third. Then (1)$\Longrightarrow$(2) follows by an argument analogous to the proof in Theorem \ref{thm 5.5}.
For (2)$\Longrightarrow$(1), we set $s= \mathrm{Gpd}_{R}X$, $i=\mathrm{inf}(X)$ and assume that $X\simeq G$ for a C-E Gorenstein projective complex $G$.
By \cite[Theorem 3.1]{CFH06}, $G_{s}/\mathrm{B}_{s}(G)$ is a Gorenstein projective module. Then $X\simeq G_{(s, i)}$ with $G_{(s, i)}$
a C-E Gorenstein projective complex. Hence $\mathrm{CE}\text{-}\mathrm{Gpd}_{R}X \leq s$.

Now let $R$ be a noetherian ring with a dualizing module $C$. It follows from \cite[Proposition 3.9]{EJOR05} that every Gorenstein
projective module is in $\mathcal{A}^{0}_{C}(R)$.
Hence the condition $\mathrm{CE}\text{-}\mathrm{Gpd}_{R}X <\infty$ implies $X\simeq G$ for a bounded C-E Gorenstein projective complex $G$,
and hence $X\simeq G\in \mathrm{CE}\text{-}\mathcal{A}_{C}(R)$. Then (1)$\Longrightarrow$(3) follows. It is immediate from Proposition \ref{prop 4.2}
that (3)$\Longrightarrow$(4), and (2)$\Longleftrightarrow$(4) follows by \cite[Theorem 4.1]{CFH06}.
\end{proof}

\begin{cor}\label{cor 5.8}
Let $X\in\mathcal{D}_{\sqsupset}(R)$ be a complex of finite $\mathrm{C}$-$\mathrm{E}$ Gorenstein projective dimension. Then
$\mathrm{CE}\text{-}\mathrm{Gpd}_{R}X = \mathrm{Gpd}_{R}X$.
\end{cor}

Dually, for a homologically left-bounded complex $X$, one can define C-E injective dimension, denoted by
$\mathrm{CE}\text{-}\mathrm{id}_{R}X$, to be
$\mathrm{inf}\{\mathrm{sup}\{i\in \mathbb{Z}\mid I_{-i} \neq 0 \} \mid X \simeq I \text{ with } I \text{ C-E injective}\}$.
Similarly, C-E Gorenstein injective dimension of $X$ is defined with
$\mathrm{inf}\{\mathrm{sup}\{i\in \mathbb{Z}\mid E_{-i} \neq 0 \} \mid X \simeq E \text{ with } E \text{ C-E Gorenstein injective}\}$,
and is denoted by  $\mathrm{CE}\text{-}\mathrm{Gid}_{R}X$.

\begin{thm}\label{thm 5.9}
Let $X\in\mathcal{D}_{\sqsubset}(R)$. Assume that  $X\simeq I$ for a $\mathrm{C}$-$\mathrm{E}$ injective complex $I$
(equivalently, $\mathrm{H}(X)$ is a complex of injective modules).
Then the following are equivalent:\\
\indent$(1)$ $\mathrm{CE}\text{-}\mathrm{id}_{R}X$ is finite.\\
\indent$(2)$ $\mathrm{id}_{R}X$, $\mathrm{id}_{R}\mathrm{Z}(X)$, $\mathrm{id}_{R}\mathrm{B}(X)$
and $\mathrm{id}_{R}\mathrm{H}(X)$ are all finite.
\end{thm}

\begin{cor}\label{cor 5.10}
Let $X\in\mathcal{D}_{\sqsubset}(R)$ be a complex of finite $\mathrm{C}$-$\mathrm{E}$ injective dimension. Then
$\mathrm{CE}\text{-}\mathrm{id}_{R}X = \mathrm{id}_{R}X$.
\end{cor}

\begin{thm}\label{thm 5.11}
Let $X\in\mathcal{D}_{\sqsubset}(R)$.
Assume that $X\simeq E$ for a $\mathrm{C}$-$\mathrm{E}$ Gorenstein injective complex $E$.
Then the following are equivalent:\\
\indent$(1)$  $\mathrm{CE}\text{-}\mathrm{Gid}_{R}X$ is finite.\\
\indent$(2)$  $\mathrm{Gid}_{R}X$, $\mathrm{Gid}_{R}\mathrm{Z}(X)$, $\mathrm{Gid}_{R}\mathrm{B}(X)$
and $\mathrm{Gid}_{R}\mathrm{H}(X)$ are all finite.

Moreover, if $R$ is a noetherian ring with a dualizing module $C$, then the above are equivalent to:\\
\indent$(3)$   $X\in \mathrm{CE}\text{-}\mathcal{B}_{C}(R)$.\\
\indent$(4)$  $X$, $\mathrm{Z}(X)$, $\mathrm{B}(X)$ and $\mathrm{H}(X)$ are in $\mathcal{B}_{C}(R)$.
\end{thm}

\begin{cor}\label{cor 5.12}
Let $X\in\mathcal{D}_{\sqsubset}(R)$ be a complex of finite $\mathrm{C}$-$\mathrm{E}$ Gorenstein injective dimension. Then
$\mathrm{CE}\text{-}\mathrm{Gid}_{R}X = \mathrm{Gid}_{R}X$.
\end{cor}

\begin{ack*}
The research was supported by National Natural Science Foundation of China (No.
11401476, 11261050) and  NWNU-LKQN-13-1.
The authors would like to thank Dr. Bo Lu, Dr. Gang Yang and Dr. Chunxia Zhang for their suggestions on the
early version of the manuscript.
\end{ack*}

\end{document}